\pdfoutput=1
\RequirePackage{ifpdf}
\ifpdf 
\documentclass[pdftex]{sigma}
\else
\documentclass{sigma}
\fi

\numberwithin{equation}{section}

\newtheorem{Theorem}{Theorem}[section]
\newtheorem{Corollary}[Theorem]{Corollary}
{ \theoremstyle{definition}
\newtheorem{Remark}[Theorem]{Remark} }

\begin{document}

\newcommand{\arXivNumber}{1707.05216}

\renewcommand{\PaperNumber}{035}

\FirstPageHeading

\ShortArticleName{On Basic Fourier--Bessel Expansions}

\ArticleName{On Basic Fourier--Bessel Expansions}

\Author{Jos\'e Luis CARDOSO}

\AuthorNameForHeading{J.L.~Cardoso}

\Address{Mathematics Department, University of Tr\'as-os-Montes e Alto Douro (UTAD),\\ Vila Real, Portugal}
\Email{\href{mailto:jluis@utad.pt}{jluis@utad.pt}}

\ArticleDates{Received September 27, 2017, in final form April 11, 2018; Published online April 17, 2018}

\Abstract{When dealing with Fourier expansions using the third Jackson (also known as Hahn--Exton) $q$-Bessel function, the corresponding positive zeros $j_{k\nu}$ and the ``shifted'' zeros,~$qj_{k\nu}$, among others, play an essential role. Mixing classical analysis with $q$-analysis we were able to prove asymptotic relations between those zeros and the ``shifted'' ones, as well as the asymptotic behavior of the third Jackson $q$-Bessel
function when computed on the ``shifted'' zeros. A version of a~$q$-analogue of the Riemann--Lebesgue theorem within the scope of basic Fourier--Bessel expansions is also exhibited.}

\Keywords{third Jackson $q$-Bessel function; Hahn--Exton $q$-Bessel function; basic Fourier--Bessel expansions; basic hypergeometric function; asymptotic behavior; Riemann--Lebesgue theorem}

\Classification{42C10; 33D45; 33D15}

\section{Introduction}

When dealing with basic Fourier--Bessel expansions, due to convergence issues, it is crucial to know the asymptotic behavior of the third Jackson $q$-Bessel function when computed in its own shifted zeros. For this purpose, in the sequel of Rahman, as pointed out by Koelink and Swarttouw \cite[p.~696]{KS}, ``the intermingling of (ordinary) analysis and $q$-analysis may be fruitful''.

In the literature, the function $J_{\nu}^{(3)}(z;q)\equiv J_{\nu}(z;q)$, where $\nu$ and $q$ are parameters satisfying $\nu>-1$ and $0<q<1$, is usually identified by the third Jackson $q$-Bessel function or by the Hahn--Exton $q$-Bessel function:
\begin{gather}\label{def}
J_{\nu }^{(3)}(z;q)\equiv J_{\nu }(z;q):=z^{\nu }\frac{\big(q^{\nu +1};q\big)_{\infty }}{(q;q)_{\infty }} \sum\limits_{k=0}^{+\infty}(-1)^{k}\frac{q^{\frac{k(k+1)}{2}}}{\big(q^{\nu +1};q\big)_{k}(q;q)_{k}}z^{2k} .
\end{gather}
Using the basic hypergeometric representation \cite[p.~4]{GR} for ${_r}\phi_s$, it is very well known that (\ref{def}) can be written as
\begin{gather}\label{def2}
J_{\nu }(z;q):=z^{\nu }\frac{\big(q^{\nu +1};q\big)_{\infty }}{(q;q)_{\infty }} {_1}\phi_1\big(0;q^{\nu+1};q,qz^2\big) .
\end{gather}

In \cite{A2} it was shown, under some restrictions, that these functions are the only ones that satisfy a $q$-analogue of the Hardy result~\cite{Ha} about functions \emph{orthogonal with respect to their own zeros}.

We have the following limit
\begin{gather*}\lim_{q\rightarrow1}J_{\nu}\left(\frac{1-q}{2}x;q\right)=J_{\nu}(x) ,\end{gather*}
where $J_{\nu}(x)$ is the (classical) Bessel function of the first kind~\cite{Watson} of order~$\nu$,{\samepage
\begin{gather*} J_{\nu}(x):= \left(\frac{x}{2}\right)^{\nu} \sum_{k=0}^{+\infty}\frac{(-1)^k\big(\frac{x}{2}\big)^{2k}}{k!\Gamma(\nu+k+1)} ,\end{gather*}
which shows that $J_{\nu }(z;q)$ is a $q$-analogue of the Bessel function $J_{\nu}(z)$.}

Exton originally in \cite{E1978,E} and later Koelink and Swarttouw in \cite[Proposition~3.5, p.~696]{KS}, proved that the function~(\ref{def}) satisfies an orthogonality of the form
\begin{gather}\label{ort}
\int_{0}^{1}xJ_{\nu }\big(qj_{n\nu}x;q^{2}\big)J_{\nu}\big(qj_{m\nu}x;q^{2}\big){\rm d}_{q}x=\eta_{n,\nu}\delta _{n,m},\\
\eta_{n,\nu}\equiv\eta_{n,\nu}(q)= \int_{0}^{1}tJ_{\nu}^2\big(qj_{k\nu}t;q^{2}\big){\rm d}_{q}t ,\nonumber
\end{gather}
where $j_{n\nu}\equiv j_{n\nu}\big(q^2\big)$, with $j_{1\nu }(q^{2})<j_{2\nu }(q^{2})<\cdots$, represent the (ordered) positive zeros of $J_{\nu}\big(z;q^2\big)$ and the $q$-integral in the interval $[0,1]$ is the one introduced by Thomae (in~1869 and in~1870)
\begin{gather}\label{qintegral-0a}
\int_{0}^{1}f(t){\rm d}_{q}t:= (1-q ) \sum_{k=0}^{+\infty }f\big(q^{k}\big)q^{k} ,
\end{gather}
which was later generalized by Jackson to any interval $[a,b]$ (see \cite[p.~19]{GR}).

Using the definition~(\ref{qintegral-0a}) we may consider an inner product by setting
\begin{gather}\label{inner-product}
\langle f,g\rangle :=\int_{0}^{1}f(t)\overline{g(t)}{\rm d}_{q}t .
\end{gather}
The resulting Hilbert space is commonly denoted by $L_{q}^{2}[0,1]$, being this space a separable Hilbert space \cite{Ann,CP}, consisting of the (quotient) set of functions $f$ such that
\begin{gather*} \int_{0}^{1}|f(t)|^2{\rm d}_{q}t<+\infty .\end{gather*}

Abreu and Bustoz showed in \cite{AB} that the sequence $\{u_{k}\}_{k}$, where
\begin{gather*}
u_{k}(x)=\frac{x^{\frac{1}{2}}J_{\nu }\big(j_{k\nu }qx;q^{2}\big)}{\big\Vert x^{\frac{1}{2}}J_{\nu }\big(j_{k\nu }qx;q^{2}\big)\big\Vert } ,
\end{gather*}
define a \emph{complete} system in $L_{q}^{2}[0,1]$, meaning that, whenever a function $f$ is in $L_{q}^{2}[0,1]$, if
\begin{gather*} \int_{0}^{1}f(x)u_{k}(x){\rm d}_{q}x=0 , \qquad k=1,2,3,\ldots ,\end{gather*}
then $f \big(q^{k}\big)=0$, $k=0,1,2,\ldots$.

Regarding criteria for completeness in $L_{q}[0,1]$ see \cite{A}, with an interesting application to a~$q$-version of the uncertainty principle via the $q$-Hankel transform and a~completeness property of the third Jackson $q$-Bessel function.

Basic Fourier expansions were studied in \cite{BC,BS}, with respect to quadratic grids and to linear grids, respectively. For an overview over basic Fourier expansions see~\cite{S2003}. In~\cite{JLC,JLC2} were presented results regarding convergence issues concerning basic Fourier expansions involving the basic sine and cosine functions considered by Suslov~\cite{S}, which are equivalent to the ones introduced by Exton~\cite{E}. With properties connected to this or related functions we refer to~\cite{St}. For new definitions of $q$-exponential and $q$-trigonometric functions see~\cite{C}. Since we are using and proving some asymptotic results, we highlight~\cite{D} on this subject, where a~complete asymptotic expansion for the $q$-Pochhammer symbol (or, the infinite $q$-shifted factorial $(z;q)_{\infty}$) was exhibited. We also point out~\cite{SS2016} with an appendix where, among others, asymptotic results for the theta function, for the ${_1}\phi_1(0;\omega;q,z)$ function and for its derivative were presented. This last two ones, with an useful separation of the terms that \emph{increase} from the terms that \emph{decrease}, were crucial to establish our results of the third Jackson $q$-Bessel function.

Other publications also show estimates or inequalities involving the third Jackson $q$-Bessel function: for instance, equation~(3.2.14) of \cite{S-PhD} when $\nu=n$, equation~(2.4) of Proposition~2.1 from~\cite{KooS}, Lemma~1 from~\cite{BBEB} and~\cite{ESF} for the particular case of third Jackson $q$-Bessel function of order zero.

In our judgement the main results of this work are Theorems~\ref{asymptotic-derivativeJ3-at-zeros},~\ref{shiftedzero},
\ref{AsymptoticJ3-at-shifted-zeros} and~\ref{Riemann--Lebesgue-2}. We also emphasize Theorem~\ref{signderivative-J3-2} since it was decisive for the proof of Theorem~\ref{shiftedzero} and mostly because of its important Corollaries~\ref{monotonyJ3-1} and \ref{monotonyJ3-2}. We believe that all the results stated in Section~\ref{asymptotic-properties-J3} and Theorem \ref{Riemann--Lebesgue-2} of Section~\ref{Riemann--Lebesgue-theorem} are original.

To know the asymptotic behaviors given by Theorems \ref{asymptotic-derivativeJ3-at-zeros} and \ref{AsymptoticJ3-at-shifted-zeros}, are decisive to study convergence properties of the basic Fourier--Bessel expansions. This is an important issue and, in our opinion, it is the most relevant contribution of the present work.

Many questions concerning basic Fourier or basic Fourier--Bessel expansions can be raised: analogues of Dirichlet's kernel, Riemann--Lebesgue theorem, Dini's condition or summability (Fej\'er's theorem) and many other topics are open problems since some of the nice properties used in the corresponding proofs are no longer valid in the context of basic expansions. Regarding this, we approach these difficulties and push a little further towards a $q$-analogue of the Riemann--Lebesgue theorem.

The paper is organized as follows: in Section~\ref{def-pre} we collect the main definitions and preliminary results that were taken from other publications; in Section~\ref{asymptotic-properties-J3} we present some asymptotic behavior of the third Jackson $q$-Bessel function and of its derivative when computed at certain points. We also study the asymptotic behavior of the zeros $j_{m\nu}$ and their relations with the ``shifted'' zeros $qj_{m\nu}$ or $\frac{j_{m\nu}}{q}$, for large values of $m=1,2,3,\ldots$ and explore its consequences to obtain other results; we finish with Section~\ref{Riemann--Lebesgue-theorem} where an analogue of the Riemann--Lebesgue theorem concerning basic Fourier--Bessel expansion is proved.

\section{Definitions and preliminary results}\label{def-pre}

Fixing $0<q<1$ and following the standard notations of \cite{AAR,GR}, the $q$-shifted factorial for a~finite positive integer~$n$ is defined by
\begin{gather*}
(a;q)_{n}= ( 1-q ) ( 1-aq )\cdots\big(1-aq^{n-1}\big)
\end{gather*}
and the zero and infinite cases as
\begin{gather*}
(a;q)_{0}=1 ,\qquad (a;q)_{\infty }=\lim\limits_{n\rightarrow \infty }(a;q)_{n} .
\end{gather*}
The third Jackson $q$-Bessel function has a countable infinite number of real and simple zeros, as it was shown in~\cite{KS}. In \cite[Theorem~2.3]{ABC} it was proved the following theorem:

\begin{Theorem}\label{TheoremA} For every $q\in {}]0,1[$, $k_{0}\in N$ exists such that, if $k\geq k_{0}$ then
\begin{gather*}
j_{k\nu }=q^{-k+\epsilon _{k}^{(\nu )}(q^{2})} ,
\end{gather*}
with
\begin{gather*}
0<\epsilon _{k}^{(\nu )}\big(q^{2}\big)<\alpha _{k}^{(\nu )}\big(q^{2}\big) ,
\qquad \text{where} \quad
\alpha _{k}^{(\nu )}\big(q^{2}\big)=\frac{\log {\big( 1-q^{2(k+\nu)}/\big(1-q^{2k}\big)\big) }}{2\log q} .
\end{gather*}
\end{Theorem}

On this subject see \cite{AM} and \cite{SS2016}. The latter one improved the accuracy of the asymptotic expression for the zeros of the basic hypergeometric function ${_1}\phi_1\big(0;\omega;q,z\big)$, which appears in the definition~(\ref{def2}) of the Hahn--Exton $q$-Bessel function.
See also~\cite{H} where Hayman obtained an expression for the asymptotic behavior of the zeros of a certain class of entire functions, which perhaps may be extended to the third Jackson $q$-Bessel function.

Using Taylor expansion it can be shown that, as $k\rightarrow \infty$,
\begin{gather}\label{asymptoticbehavior}
\alpha _{k}^{(\nu )}\big(q^{2}\big)=\mathcal{O}\big(q^{2k}\big) .
\end{gather}

Formally, the $q$-Fourier Bessel series associated with a function $f$, by the orthogonal relation~(\ref{ort}), is defined by
\begin{gather*}
S_{q}^{\nu}[f](x):=\sum_{k=1}^{\infty }b_{k}^{(\nu)}\left(f\right)x^{\frac{1}{2}} J_{\nu}\big(qj_{k\nu }x;q^{2}\big) ,
\end{gather*}
with the coefficients $b_{k}^{(\nu)}$ given by
\begin{gather*}
b_{k}^{(\nu)}(f) =\frac{1}{\eta_{k,\nu}} \int_{0}^{1}t^{\frac{1}{2}}f(t)J_{\nu }\big(qj_{k\nu}t;q^{2}\big){\rm d}_{q}t ,
\end{gather*}
or, which we rather prefer,
\begin{gather}\label{q-FourierSeries}
S_{q}^{(\nu)}[f](x):=\sum_{k=1}^{+\infty }a_{k}^{(\nu)} (f ) J_{\nu}\big(qj_{k\nu }x;q^{2}\big) ,
\end{gather}
with the coefficients $a_{k}^{(\nu)}$ given by
\begin{gather}\label{q-FCoefficient}
a_{k}^{(\nu)}(f) =\frac{1}{\eta_{k,\nu}}\int_{0}^{1}tf(t)J_{\nu }\big(qj_{k\nu}t;q^{2}\big){\rm d}_{q}t
\end{gather}
and $\eta _{k,\nu}$ by
\begin{gather}
 \int_{0}^{1}tJ_{\nu}^2\big(qj_{k\nu}t;q^{2}\big){\rm d}_{q}t = \frac{q-1}{2}q^{\nu -1}J_{\nu +1}\big(qj_{k\nu
};q^{2}\big)J_{\nu}^{\prime}\big(j_{k\nu };q^{2}\big) \nonumber\\
 \hphantom{\int_{0}^{1}tJ_{\nu}^2\big(qj_{k\nu}t;q^{2}\big){\rm d}_{q}t }{} = \frac{q-1}{2j_{k\nu}}q^{\nu-2}J_{\nu}\big(qj_{k\nu};q^{2}\big)J_{\nu}^{\prime}\big(j_{k\nu};q^{2}\big) ,\label{eta}
\end{gather}
where the last equality can be derived from \cite[Proposition~3.5]{KS} or \cite[Proposition~5]{JLC3}.

The asymptotic behavior of the $q$-integral that appears in the Fourier coefficient~(\ref{q-FCoefficient}), as well as the asymptotic behavior (as $k\to\infty$) of the factors $J_{\nu}\big(qj_{k\nu};q^{2}\big)$ and $J_{\nu}^{\prime}\big(j_{k\nu};q^{2}\big)$ appearing in~$\eta _{k,\nu}$, are crucial for further developments related with convergence issues of the Fourier--Bessel expansion~(\ref{q-FourierSeries}). To study those behaviors will be the main purpose in this work.

Using the expansion obtained by Olde Daalhuis \cite[equation~(3.13), p.~905]{D} for the (infinite) $q$-shifted factorial (or $q$-Pochhammer symbol), \v{S}tampach and \v{S}\v{t}ov\'i\v{c}ek~\cite{SS2016} rewrote it in the following clearer form: considering the notation
\begin{gather}\label{A.1}
\tilde{q}=e^{\frac{4\pi^2}{\ln{(q)}}} ,\qquad\beta(z)=\frac{\pi\ln{(z)}}{\ln{(q)}}
\end{gather}
and
\begin{gather}\label{A.2}
A(z)=2q^{-\frac{1}{12}}\sqrt{z}e^{-\frac{\ln^2{(z)}}{2\ln{(q)}}+\frac{\pi^2}{3\ln{(q)}}}
\big|\big(\tilde{q}e^{2i\beta(z)};\tilde{q}\big)_{\infty}\big|^2
\end{gather}
then
\begin{gather*}
(z;q)_{\infty}=\frac{A(z)}{\big(\frac{q}{z};q\big)_{\infty}}\sin{(\beta(z))},
\end{gather*}
where $z>0$.
Using a symmetric relation \cite[equation~(2.3), p.~448]{KooS} satisfied by the basic function ${_1}\phi_1(0;\omega;q,z)$, Olde Daalhuis \cite[pp.~907--908]{D} describes briefly how to obtain an asymptotic expansion for the function
\begin{gather*}J_{\nu}\big(z;q^{2}\big)=z^{\nu}\frac{(q^{\nu+1};q)_{\infty}}{(q;q)_{\infty}}
{_1}\phi_1\big(0;q^{2(\nu+1)};q,qz^2\big).\end{gather*}
Later, \v{S}tampach and \v{S}\v{t}ov\'i\v{c}ek \cite{SS2016}, proved the following theorem which displays an asymptotic behavior for the function ${_1}\phi_1(0;\omega;q,z)$ and for its derivative, as $z\to\infty$.

\begin{Theorem} \label{TheoremB} Let $K(z):=\big[\frac 1 2-\frac{\ln{(z)}}{\ln{(q)}}\big]$ where $[x]$ represents the integer part of $x\in\mathbb{R}$. With the notation \eqref{A.1}, \eqref{A.2} and assuming that $0\leq \omega<1$, there exist functions $B(\omega,z)$ and $C(\omega,z)$ such that
\begin{gather*}{_1}\phi_1(0;\omega;q,z)=\frac{B(\omega,z)}{(\omega;q)_{\infty}}\\
\qquad{} \times \left(A(z)\sin{(\beta(z))}+(-1)^{K(z)+1}q^{\frac{(K(z)+1)K(z)}{2}}
\omega^{K(z)+1}\frac{\big(q^{K(z)+1}z;q\big)_{\infty}}{(q;q)_{\infty}}C(\omega,z)\right),
\end{gather*}
where, for $\omega$ fixed, $B(\omega,z)=1+O\big(z^{-1}\big)$, $C(\omega,z)=1+O\big(z^{-1}\big)$ as $z\to+\infty$.
\end{Theorem}

\begin{Theorem} \label{TheoremC} Under the same assumptions of the previous theorem,
\begin{gather*}
 \frac{\partial {_1}\phi_1(0;\omega;q,z)}{\partial z} = \frac{A(z)}{(\omega;q)_{\infty}z} \left(\left(-\frac{\beta(z)}{\pi}+\frac{1}{2}\right)\sin{(\beta(z))}+ \frac{\pi}{\ln{(q)}}\cos{(\beta(z))} \right. \\
\left.\hphantom{\frac{\partial {_1}\phi_1(0;\omega;q,z)}{\partial z} =}{} +\frac{8\pi}{\ln{(q)}}\sum_{k=1}^{\infty} \frac{\tilde{q}^k}{\big|1-\tilde{q}^ke^{-2i\beta(z)}\big|^2} \sin^2{(\beta(z))}\cos{(\beta(z))}+O\left(\frac{\ln{(z)}}{z}\right)\right),
\end{gather*}
as $z\to+\infty$.
\end{Theorem}

\section[Asymptotic properties of the function $J_{\nu }\big(z;q^2\big)$ and its derivative]{Asymptotic properties of the function $\boldsymbol{J_{\nu }\big(z;q^2\big)}$\\ and its derivative}\label{asymptotic-properties-J3}

Theorem~2 of \cite[p.~12]{JLC3},
\begin{gather*}\big| J_{\nu }\big(qj_{k\nu};q\big)\big|\leq A_{\nu}(q)
q^{-\big(k+\frac{\nu-2}{2}-\epsilon_k^{(\nu)}\big)^2} , \qquad \text{where} \quad A_{\nu}(q)>0 ,\end{gather*}
establishes a superior bound for the asymptotic behavior of $J_{\nu}\big(qj_{k\nu};q^{2}\big)$ as $k\to\infty$.

We notice that this bound can be enlarged for the cases{\samepage
\begin{gather*}
\big| J_{\mu }\big(qj_{k\nu};q\big)\big|\leq B_{\mu}(q)
q^{-\big(k+\frac{\mu-3}{2}-\epsilon_k^{(\nu)}\big)^2} , \qquad \text{where} \quad B_{\mu}(q)
=\frac{q^{\frac{\mu}{2}\big(\frac{\mu}{2}-1\big)}}{\big(1-q^2\big)\big(q^2;q^2\big)^2_{\infty}} .
\end{gather*}
Its prove is essentially coincident with the corresponding one of the \cite{JLC3} so we omit it.}

However, at least when $\mu=\nu$ or $\mu=\nu+1$, the above estimate for $J_{\mu}\big(qj_{k\nu};q\big)$ does not seem accurate. As a matter of fact, by Theorem~\ref{TheoremA} and (\ref{asymptoticbehavior}), as $k\to\infty$, the product $qj_{k\nu}$ is ``closed'' to the positive zero $ j_{k-1,\nu}$ of the function $J_{\nu }\big(z;q^2\big)$ and, as a consequence, we expect $J_{\nu}\big(qj_{k\nu};q^{2}\big)$ to approach zero. So we look for a better bound when $\mu=\nu$ or $\mu=\nu+1$.

Because of the basic hypergeometric representation~(\ref{def2}), in order to keep the results more general and applicable to other situations, at the final of this section we present a subsection with the corresponding main results for the function ${_1}\phi_1(0;\omega;q,z)$.

\subsection[Asymptotic properties of $J_{\nu}^{\prime}\big(z;q^2\big)$]{Asymptotic properties of $\boldsymbol{J_{\nu}^{\prime}\big(z;q^2\big)}$}
During this subsection and to avoid any confusion, most of the times we prefer to use $\frac{\partial J_{\nu }(z;q^2)}{\partial z}$ rather than $J_{\nu}^{\prime}\big(z;q^2\big)$.

\begin{Theorem}\label{signderivative-J3-2}
Let $\{\theta_m\}_m$ be a sequence such that $0\leq\theta_m<1$ for $m=1,2,3,\ldots$:
\begin{enumerate}\itemsep=0pt
\item[$(i)$] if $\lim\limits_{m\to\infty}m\theta_m=0$ then $\operatorname{sgn}\Big(\frac{\partial J_{\nu}(z;q^2)}{\partial z}_{\big|z=q^{-m+\theta_m}}\Big)=(-1)^{m}$;
\item[$(ii)$] if $\lim\limits_{m\to\infty}m\theta_m=\infty$ then $\operatorname{sgn}\Big(\frac{\partial J_{\nu }(z;q^2)}{\partial z}_{\big|z=q^{-m+\theta_m}}\Big)=(-1)^{m-1}$, being both signs valid for large values of~$m$.
\end{enumerate}
\end{Theorem}
\begin{proof} By the definition of the Hahn--Exton $q$-Bessel function~(\ref{def}) we may write
\begin{gather*}
J_{\nu}(z;q^2)= \frac{\big(q^{2(\nu+1)};q^2\big)_{\infty}}{\big(q^2;q^2\big)_{\infty}}z^{\nu}{_1}\phi_1\big(0;q^{2(\nu+1)};q^2,q^2z^2\big).
\end{gather*}
Computing its derivative one gets
\begin{gather*}
\frac{\partial J_{\nu }\big(z;q^2\big)}{\partial z}= \frac{\big(q^{2(\nu+1)};q^2\big)_{\infty}}{\big(q^2;q^2\big)_{\infty}} \\
\hphantom{\frac{\partial J_{\nu }\big(z;q^2\big)}{\partial z}=}{}\times
\left(\nu z^{\nu-1}{_1}\phi_1\big(0;q^{2(\nu+1)};q^2,q^2z^2\big)+ 2q^2z^{\nu+1}\frac{\partial {_1}\phi_1\big(0;q^{2(\nu+1)};q^2,y\big)}{\partial y}_{\big|y=q^{2}z^2}\right).
\end{gather*}

Now, by Theorems~\ref{TheoremB} and~\ref{TheoremC}, with the notation $\omega=q^{2(\nu+1)}$ and~(\ref{A.1}), (\ref{A.2}) with $q$ shifted to $q^2$, as $z\to+\infty$, we have
\begin{gather}
\frac{\partial J_{\nu }\big(z;q^2\big)}{\partial z}\equiv J_{\nu}^{\prime}\big(z;q^2\big)
=\frac{z^{\nu-1}}{(q^2;q^2)_{\infty}}\Bigg\{A\big(q^2z^2\big)\Bigg\{\left(\nu B\big(\omega,q^2z^2\big)-
\frac{2}{\pi}\beta\big(q^2z^2\big)+1\right)\nonumber\\
\hphantom{\frac{\partial J_{\nu }\big(z;q^2\big)}{\partial z}\equiv}{} \times \sin \big(\beta\big(q^2z^2\big)\big)+\frac{\pi}{\ln{q}}\cos{\big(\beta\big(q^2z^2\big)\big)} \nonumber\\
\hphantom{\frac{\partial J_{\nu }\big(z;q^2\big)}{\partial z}\equiv}{} +\frac{8\pi}{\ln{q}}\!\sum_{k=1}^{\infty}\!\frac{
\tilde{q}^k}{\big|1-\tilde{q}^ke^{-2i\beta(q^2z^2)}\big|^2}
\sin^2{\big(\beta\big(q^2z^2\big)\big)}\cos{\big(\beta\big(q^2z^2\big)\big)}+
O\!\left(\frac{\ln\big(q^2z^2\big)}{q^2z^2}\right)\!\Bigg\}\!\nonumber \\
\hphantom{\frac{\partial J_{\nu }\big(z;q^2\big)}{\partial z}\equiv}{}
+(-1)^{K\big(q^2z^2\big)+1} q^{\big(K\big(q^2z^2\big)+1\big)K\big(q^2z^2\big)} \omega^{K\big(q^2z^2\big)+1}\nu B\big(\omega,q^2z^2\big)\nonumber\\
\hphantom{\frac{\partial J_{\nu }\big(z;q^2\big)}{\partial z}\equiv}{}
\times \frac{\big(q^{2K(q^2z^2)+4}z^2;q^2\big)_{\infty}}{\big(q^2;q^2\big)_{\infty}}C\big(\omega,q^2z^2\big)\Bigg\} .\label{derivative-J3}
\end{gather}
Taking into account that, shifting $q$ to $q^2$ and putting $z=q^{-m+\theta_m}$ for $m\in\mathbb{N}$, $\beta\big(q^2z^2\big)=(-m+1+\theta_m)\pi$, $K\big(q^2z^2\big)=\big[m-\frac{1}{2}-\theta_m\big]$, identity~(\ref{derivative-J3}) gives, as $m\to\infty$,
\begin{gather}
 \frac{\partial J_{\nu }\big(z;q^2\big)}{\partial z}_{\big|z=q^{-m+\theta_m}}=
\frac{q^{(-m+\theta_m)(\nu-1)}}{\big(q^2;q^2\big)_{\infty}}\Bigg\{A\big(q^{2-2m+2\theta_m}\big)(-1)^{m-1} \nonumber\\
\qquad {} \times
\Bigg\{\big(\nu B\big(\omega,q^{2-2m+2\theta_m}\big)+2m-1-2\theta_m\big)
\sin{(\pi\theta_m)}+\frac{\pi}{\ln{q}}\cos{(\pi\theta_m)}\nonumber\\
\qquad{} +\frac{8\pi}{\ln{q}} \sum_{k=1}^{\infty}\frac{\tilde{q}^k}{\big|1-\tilde{q}^ke^{-2i\theta_m\pi}\big|^2}
\sin^2{ (\pi\theta_m )}\cos{ (\pi\theta_m )}
+O\left(\frac{\ln\big(q^{2-2m+2\theta_m}\big)}{q^{2-2m+2\theta_m}}\right)\Bigg\}\nonumber\\
\qquad {} +(-1)^{[m-\frac{1}{2}-\theta_m]+1}q^{([m-\frac{1}{2}-\theta_m]+1)[m-\frac{1}{2}-\theta_m]}
\omega^{[m-\frac{1}{2}-\theta_m]+1}\nonumber \\
\qquad{} \times\nu B\big(\omega,q^{2-2m+2\theta_m}\big)\frac{\big(q^{2[m-\frac{1}{2}-\theta_m]+4-2m+2\theta_m};q^2\big)_{\infty}}{\big(q^2;q^2\big)_{\infty}}
C\big(\omega,q^{2-2m+2\theta_m}\big)\Bigg\} .\label{derivativeJ3-at-q}
\end{gather}
We notice that, for large values of $m$,
\begin{gather*} A\big(q^{-m+\theta_m}\big)=2 q^{-\frac{(m+1)m}{2}+m\theta_m+\frac{\theta_m(1-\theta_m)}{2}+\frac{\pi^2}{3\ln{(q)}}-\frac{1}{12}}
\big|\big(e^{2i\pi\theta_m}e^{\frac{4\pi^2}{\ln{(q)}}}; e^{\frac{4\pi^2}{\ln{(q)}}}\big)_{\infty}\big|^2,
\end{gather*} hence
\begin{gather}
A\big(q^{2-2m+2\theta_m}\big) = 2q^{-m(m-1)+2(m-1)\theta_m+\theta_m(1-\theta_m)+\frac{2\pi^2}{3\ln{(q)}}-\frac{1}{6}}
\big|\big(e^{2i\pi\theta_m}e^{\frac{2\pi^2}{\ln{(q)}}};
e^{\frac{2\pi^2}{\ln{(q)}}}\big)_{\infty}\big|^2>0\!\!\!\!\label{Aq2}
\end{gather}
and, by Theorem~\ref{TheoremB}, $B(\omega,z)=1+O\big(z^{-1}\big)$, $C(\omega,z)=1+O\big(z^{-1}\big)$ as $z\to+\infty$. We also
note that $\big[m-\frac{1}{2}-\theta_m\big]$ equals $m-1$ or $m-2$.

Now, from (\ref{derivativeJ3-at-q}), we conclude the following:

On one hand, if $\lim\limits_{m\to\infty}m\theta_m=0$ then, as $m\to\infty$, the dominant term of the sign of~(\ref{derivativeJ3-at-q}) is
$(-1)^{m-1}\frac{\pi}{\ln{(q)}}\cos{(\pi\theta_m)}$, with $\frac{\pi}{\ln{(q)}}<0$. This proves part~(i) of the theorem.

On the other hand, if $\lim\limits_{m\to\infty}m\theta_m=\infty$ then, as $m\to\infty$, the dominant term for the sign turns to be
$(-1)^{m-1}(2m-1-2\theta_m)\sin{(\pi\theta_m)}$, which proves part~(ii).
\end{proof}

\begin{Remark}The assumption of Theorem \ref{signderivative-J3-2} requiring that $0\leq\theta_m<1$ for $m=1,2,3,\ldots$, can be weakened
(with minor changes in the corresponding proof) to $0\leq\theta_m<1$ for sufficient large values of~$m$.

Also, when the sequence $\{m\theta_m\}_{m}$ converges to a strictly positive real number, or when it is a bounded but not convergent sequence, then it is also possible to state conditions in order to obtain conclusions.
\end{Remark}

We notice that if $\{\theta_m^*\}_m$ is any sequence which satisfies $0<\theta_m^*<1$ for all $m=1,2,3,\ldots$ and $\lim\limits_{m\to\infty}m\theta_m^*=0$ then, by (\ref{derivativeJ3-at-q}), we conclude that part~(i) of the Theorem~\ref{signderivative-J3-2} remains true for every other sequence $\{\gamma_m\}_m$ such that $0<\gamma_m\leq\theta_m^*$. This implies the next result.

\begin{Corollary}\label{monotonyJ3-1} Let $\{\theta_m^*\}_m$, with $0<\theta_m^*<1$, be a sequence such that $\lim\limits_{m\to\infty}m\theta_m^*=0$. Then, for large values of $m$, the sign of $\frac{\partial J_{\nu }(z;q^2)}{\partial z}$ remains constant in each interval $\big]q^{-m+\theta_m^{*}},q^{-m}\big[$.
\end{Corollary}

In particular, because of (\ref{asymptoticbehavior}), it follows immediately the following corollary.
\begin{Corollary}\label{monotonyJ3-2}
Considering $\theta_m^*=\alpha_m^{(\nu)}$, for $m=1,2,3,\ldots$, of Theorem~{\rm \ref{TheoremA}} then, for large values of~$m$,
the sign of $\frac{\partial J_{\nu }(z;q^2)}{\partial z}$ remains constant in each interval $\big]q^{-m+\alpha_m^{(\nu)}},q^{-m}\big[$.
\end{Corollary}
We end this subsection with the following theorem.
\begin{Theorem}\label{asymptotic-derivativeJ3-at-zeros}
For large values of $m$,
\begin{gather*}\frac{\partial J_{\nu}\big(z;q^2\big)}{\partial z}_{\big|z=j_{m\nu}}
\equiv J_{\nu}^{\prime}\big(j_{m\nu};q^{2}\big)=
O\big(q^{-m(m+\nu-2)}\big) ,\qquad \text{as}\quad m\to\infty .\end{gather*}
\end{Theorem}
\begin{proof}
By Theorem~\ref{TheoremA}, consider $j_{m\nu}=q^{-m+\epsilon_m^{(\nu)}}$, where $0<\epsilon_m^{(\nu)}<\alpha_m^{(\nu)}$, and replace $\theta_m$ by~$\epsilon_m^{(\nu)}$ in~(\ref{derivativeJ3-at-q}) and~(\ref{Aq2}).

By (\ref{asymptoticbehavior}) we have $\lim\limits_{m\to\infty}m\alpha_m^{(\nu)}=0$ hence, by Theorem~\ref{TheoremA}, we also have $\lim\limits_{m\to\infty}m\epsilon_m^{(\nu)}=0$. Furthermore, taking into consideration that
\begin{gather}\label{Aq2epsilon}
A\big(q^{2-2m+2\epsilon_m^{(\nu)}}\big)=2q^{-m(m-1)+2(m-1)\epsilon_m^{(\nu)}+\epsilon_m^{(\nu)}(1-\epsilon_m^{(\nu)})+
\frac{2\pi^2}{3\ln{(q)}}-\frac{1}{6}}\big|\big(e^{2i\pi\epsilon_m^{(\nu)}}e^{\frac{2\pi^2}{\ln{(q)}}};
e^{\frac{2\pi^2}{\ln{(q)}}}\big)_{\infty}\big|^2
\end{gather}
then, for large values of $m$, the resulting dominant term of $J_{\nu}^{\prime}\big(j_{m\nu};q^{2}\big)$ from~(\ref{derivativeJ3-at-q}) is
\begin{gather}\label{dominant-term}
\frac{q^{(-m+\epsilon_m^{(\nu)})(\nu-1)}}{\big(q^2;q^2\big)_{\infty}}A\big(q^{2-2m+2\epsilon_m^{(\nu)}}\big)
(-1)^{m-1}\frac{\pi}{\ln{q}}\cos{\big(\pi\epsilon_m^{(\nu)}\big)} .
\end{gather}
Introducing (\ref{Aq2epsilon}) into (\ref{dominant-term}) and, again, using $\lim_{m\to\infty}m\alpha_m^{(\nu)}=0$, then we immediately conclude that
\begin{gather*}
\frac{\partial J_{\nu}\big(z;q^2\big)}{\partial z}_{\big|z=j_{m\nu}}\equiv
J_{\nu}^{\prime}\big(j_{m\nu};q^{2}\big)=O\big(q^{-m(m+\nu-2)}\big) ,\qquad \text{as}\quad m\to\infty .\tag*{\qed}
\end{gather*}\renewcommand{\qed}{}
\end{proof}

\subsection[Behavior of $J_{\nu }\big(qj_{k\nu};q^2\big)$]{Behavior of $\boldsymbol{J_{\nu }\big(qj_{k\nu};q^2\big)}$}

We begin this subsection by quoting the following theorem, where $j_{k\nu}$, $\epsilon_{k}^{(\nu)}\equiv\epsilon_{k}^{(\nu)}\big(q^2\big)$ and $\alpha_{k}^{(\nu)}\equiv\alpha_{k}^{(\nu)}\big(q^2\big)$ respects the notations of~(\ref{ort}) and Theorem~\ref{TheoremA}.
\begin{Theorem}\label{shiftedzero}
For large values of $k$,
\begin{gather*}qj_{k\nu}\in {}\big ] j_{k-1,\nu} , q^{-k+1} \big[ .\end{gather*}
\end{Theorem}
\begin{proof}From (\ref{ort}) with $m=n=k$ and by~(vii) of Proposition~5 \cite[p.~330]{JLC3}, we get
\begin{gather}\label{Ineq1}
J_{\nu}\big(qj_{k\nu};q^2\big)J_{\nu}^{\prime}\big(j_{k\nu};q^2\big)<0 .
\end{gather}
However, by (\ref{asymptoticbehavior}),
$\lim\limits_{k\to\infty}k\alpha_k^{(\nu)}=0$, hence one may
conclude, by Theorem \ref{signderivative-J3-2} and Corollary
\ref{monotonyJ3-2}, that the sign of
\begin{gather*}
\frac{\partial J_{\nu}\big(z;q^2\big)}{\partial z}= J_{\nu}^{\prime}\big(z;q^2\big)
\end{gather*}
in the interval $\big]q^{-k+\alpha_k^{(\nu)}},q^{-k}\big[$ is the opposite to the sign in $\big]q^{-k+1+\alpha_{k-1}^{(\nu)}},q^{-k+1}\big[$, for large values of~$k$.

Thus, for large values of $k$, by Theorem~\ref{TheoremA},
\begin{gather}\label{Ineq2}
J_{\nu}^{\prime}\big(j_{k\nu};q^2\big)J_{\nu}^{\prime}\big(j_{k-1,\nu};q^2\big)<0 .
\end{gather}
Using, now, (\ref{Ineq1}) and (\ref{Ineq2}) we may write, for large values of $k$,
\begin{gather*}
J_{\nu}\big(qj_{k\nu};q^2\big)J_{\nu}^{\prime}\big(j_{k-1,\nu};q^2\big)>0 .
\end{gather*}
This guarantees that, for large values of $k$,
\begin{gather*}
qj_{k,\nu}>j_{k-1,\nu} ,
\end{gather*}
which proves the theorem since, trivially by Theorem~\ref{TheoremA},
$qj_{k,\nu}=q^{1-k+\epsilon_k^{(\nu)}}<q^{1-k}$.
\end{proof}

The following corollaries are immediate consequences of the previous theorem.
\begin{Corollary} For large values of $k$, the sequence $\big\{\epsilon_{k}^{(\nu)}\big\}_{k}$ that appears in Theorem~{\rm \ref{TheoremA}} is strictly decreasing, i.e., there exists a positive integer $k_0$ such that $\epsilon_{k+1}^{(\nu)}<\epsilon_{k}^{(\nu)}$ whenever $k\geq k_0$.
\end{Corollary}
\begin{proof} The previous theorem guarantees that, for large values of $k$, $qj_{k\nu}>j_{k-1,\nu}$, which is equivalent to $q^{1-k+\epsilon_k^{(\nu)}}>q^{1-k+\epsilon_{k-1}^{(\nu)}}$, hence $\epsilon_k^{(\nu)}<\epsilon_{k-1}^{(\nu)}$ for large values of $k$.
\end{proof}

\begin{Corollary} For large values of $k$,
\begin{gather*}\frac{j_{k\nu}}{q}\in {}\big] q^{-k-1+\alpha_{k+1}^{(\nu)}} , j_{k+1,\nu} \big[ .\end{gather*}
\end{Corollary}
\begin{proof}
According to Theorem~\ref{TheoremA} we have both $j_{k\nu}= q^{-k+\epsilon_k^{(\nu)}}$ and $j_{k+1,\nu} = q^{-k-1+\epsilon_{k+1}^{(\nu)}}$. Therefore, $\frac{j_{k\nu}}{q}= q^{-k-1+\epsilon_{k}^{(\nu)}}$. However, since by the previous corollary, there exists a positive integer $k_0$ such that $\epsilon_{k+1}^{(\nu)}<\epsilon_{k}^{(\nu)}$ whenever $k\geq k_0$, then $q^{\epsilon_k^{(\nu)}}<q^{\epsilon_{k+1}^{(\nu)}}$ whenever $k\geq k_0$, hence it follows that $\frac{j_{k\nu}}{q}<j_{k+1,\nu}$ and, by Theorem~\ref{TheoremA}, $\frac{j_{k\nu}}{q}>q^{-k-1+\alpha_k^{(\nu)}}$, both for large values of $k$.
\end{proof}

We now prove the following theorem.
\begin{Theorem}\label{AsymptoticJ3-at-shifted-zeros} For large values of $k$,
\begin{gather*}\big|J_{\nu}\big(qj_{k\nu};q^2\big)\big|\leq
\frac{\big({-}q^2,-q^{2(\nu+1)};q^2\big)_{\infty}}{\big(q^2;q^2\big)_{\infty}} q^{(k+\nu)(k-1)} .\end{gather*}
\end{Theorem}
\begin{proof}On one hand, being $j_{k\nu}$, for $k=1,2,3,\ldots$, the positive zeros of the Hahn--Exton $q$-Bessel function, we have
\begin{gather}\label{1}
J_{\nu}\big(j_{k-1,\nu};q^2\big)=0, \qquad k=2,3,4,\ldots .
\end{gather}
On the other hand, by \cite[equation~(12), p.~1205]{BBEB},
\begin{gather}\label{2}
\big|J_{\nu}\big(q^{-k+1};q^2\big)\big|\leq \frac{\big({-}q^2,-q^{2(\nu+1)};q^2\big)_{\infty}}{\big(q^2;q^2\big)_{\infty}} q^{(k+\nu)(k-1)} .
\end{gather}
This last result was first presented in \cite{KooS} and it can also be obtained in an equivalent form using directly Theorem~\ref{TheoremB}.

 Notice that, by Theorem~\ref{TheoremA}, $j_{k-1,\nu}=q^{-k+1+\epsilon_{k-1}^{(\nu)}}$ where, as a~consequence of~(\ref{asymptoticbehavior}), $\lim\limits_{k\to\infty}(k-1)\epsilon_{k-1}^{(\nu)}=0$. Thus, by Corollary~\ref{monotonyJ3-1}, $J_{\nu}\big(z;q^2\big)$ is strictly monotone in each interval $]j_{k-1,\nu} , q^{-k+1}[$, for large values of~$k$. Now, since by Theorem~\ref{shiftedzero}, $qj_{k,\nu}\in {}]j_{k-1,\nu} , q^{-k+1}[$, then, using~(\ref{1}) and~(\ref{2}), the theorem follows.
\end{proof}

\subsection[Corresponding properties for the function ${_1}\phi_1(0;\omega;q,z)$]{Corresponding properties for the function $\boldsymbol{{_1}\phi_1(0;\omega;q,z)}$}

\begin{Theorem}Let $\omega$ be fixed in $[0,1[$ and $\{\tau_m\}_m$ be a sequence such that $0\leq\tau_m<1$ for $m=1,2,3,\ldots$:
\begin{enumerate}\itemsep=0pt
\item[$(i)$] if $\lim\limits_{m\to\infty}m\tau_m=0$ then $\operatorname{sgn}\Big(\frac{\partial {_1}\phi_1(0;\omega;q,z)}{\partial
z}_{\big|z=q^{-m+\tau_m}}\Big)=(-1)^{m+1}$;

\item[$(ii)$] if $\lim\limits_{m\to\infty}m\tau_m=\infty$ then $\operatorname{sgn}\Big(\frac{\partial {_1}\phi_1(0;\omega;q,z)}{\partial z}_{\big|z=q^{-m+\tau_m}}\Big)=(-1)^{m}$, being both signs valid for large values of~$m$.
\end{enumerate}
\end{Theorem}
\begin{proof}
Considering $z=q^{-m+\tau_m}$ in (\ref{A.1}) and (\ref{A.2}) one obtains, respectively,
\begin{gather*}
 \beta\big(q^{-m+\tau_m}\big)=(-m+\tau_m)\pi
\end{gather*}
and
\begin{gather*}
A\big(q^{-m+\tau_m}\big)=2 q^{-\frac{(m+1)m}{2}+m\tau_m+\frac{\tau_m(1-\tau_m)}{2}+\frac{\pi^2}{3\ln{(q)}}-\frac{1}{12}}
\big|\big(e^{2i\pi\tau_m}e^{\frac{4\pi^2}{\ln{(q)}}};e^{\frac{4\pi^2}{\ln{(q)}}}\big)_{\infty}\big|^2.
\end{gather*}
Then, Theorem~\ref{TheoremC} enables one to write
\begin{gather*}
\frac{\partial {_1}\phi_1(0;\omega;q,z)}{\partial z}_{\big|z=q^{-m+\tau_m}} =C_q(\omega)q^{-\frac{m(m-1)}{2}+(m-1)\tau_m+\frac{(\tau_m+1)\tau_m}{2}}\\
\qquad{} \times\Bigg\{\left(m-\tau_m+\frac{1}{2}\right)\sin{(\pi(-m+\tau_m))}+\frac{\pi}{\ln{(q)}}\cos{(\pi(-m+\tau_m))} \\
\qquad{} +\frac{8\pi}{\ln{(q)}} \sum_{k=1}^{\infty}
\frac{\tilde{q}^k}{\big|1-\tilde{q}^ke^{-2i (\pi(-m+\tau_m) )}\big|^2}\sin^2{ (\pi(-m+\tau_m) )}\cos{ (\pi(-m+\tau_m) )} \\
\qquad{} +O\left(\frac{\ln{(q^{-m+\tau_m})}}{q^{-m+\tau_m}}\right)\Bigg\}
\end{gather*}
as $m\to+\infty$, or equivalently,
\begin{gather}
 \frac{\partial {_1}\phi_1(0;\omega;q,z)}{\partial z}_{\big|z=q^{-m+\tau_m}} =
C_q(\omega)q^{-\frac{m(m-1)}{2}+(m-1)\tau_m+\frac{(\tau_m+1)\tau_m}{2}}(-1)^m \nonumber\\
\hphantom{\frac{\partial {_1}\phi_1(0;\omega;q,z)}{\partial z}_{\big|z=q^{-m+\tau_m}} =}{} \times\Bigg\{\left(m-\tau_m+\frac{1}{2}\right) \sin{(\pi\tau_m)}+ \frac{\pi}{\ln{(q)}}\cos{(\pi\tau_m)} \nonumber\\
 \hphantom{\frac{\partial {_1}\phi_1(0;\omega;q,z)}{\partial z}_{\big|z=q^{-m+\tau_m}} =}{}
+\frac{8\pi}{\ln{(q)}}\sum_{k=1}^{\infty}\frac{\tilde{q}^k}{\big|1-\tilde{q}^ke^{-2i\pi\tau_m}\big|^2}
\sin^2{(\pi\tau_m)}\cos{(\pi\tau_m)} \nonumber\\
 \hphantom{\frac{\partial {_1}\phi_1(0;\omega;q,z)}{\partial z}_{\big|z=q^{-m+\tau_m}} =}{} +O\left(\frac{\ln{\big(q^{-m+\tau_m}\big)}}{q^{-m+\tau_m}}\right)\Bigg\}\label{B1}
\end{gather}
as $m\to+\infty$, where $C_q(\omega)>0$.

On one hand we have that, if $\lim\limits_{m\to\infty}m\tau_m=0$ then, as $m\to\infty$, the dominant term of the sign of (\ref{B1}) is $(-1)^m\frac{\pi}{\ln{(q)}}\cos{(\pi\tau_m)}$, with $\frac{\pi}{\ln{(q)}}$ negative. This proves part (i) of the theorem.

On the other hand, if $\lim\limits_{m\to\infty}m\tau_m=\infty$ then, as $m\to\infty$, the dominant term turns to be $(-1)^m\big(m-\tau_m+\frac{1}{2}\big)\sin{(\pi\tau_m)}$. This proves part (ii) of the theorem.
\end{proof}

We notice that if $\{\tau_m\}_m$ is any sequence satisfying the condition (i) of the previous theorem then, by (\ref{B1}), the same conclusion of part~(i) remains true for any other sequence~$\{\gamma_m\}_m$ such that $0\leq \gamma_m\leq\tau_m$. This implies the next result.
\begin{Corollary}
Let $\{\tau_m\}_m$ be a sequence such that $\lim\limits_{m\to\infty}m\tau_m=0$. Then, for large values of~$m$, the sign of
\begin{gather*}\frac{\partial
{_1}\phi_1(0;\omega;q,z)}{\partial z}\end{gather*} remains constant in each interval $]q^{-m+\tau_m},q^{-m}[$.
\end{Corollary}

\section[$q$-analogue of the Riemann--Lebesgue theorem]{$\boldsymbol{q}$-analogue of the Riemann--Lebesgue theorem}\label{Riemann--Lebesgue-theorem}

Following the framework of \cite{CP}, we rewrite the system $\{u_m\}_{m}$ mentioned in the introduction as
\begin{gather*}u_{m}(x)=\frac{x^{\frac{1}{2}}J_{\nu }\big(j_{m\nu }qx;q^{2}\big)}{\big\Vert x^{\frac{1}{2}}J_{\nu }\big(j_{m\nu}qx;q^{2}\big)\big\Vert_{L_q^{2}[0,1]}} ,
\end{gather*}
where, by (\ref{eta}),
 \begin{gather*}
 \eta_{m,\nu} = \big\Vert x^{\frac{1}{2}}J_{\nu}\big(j_{m\nu}qx;q^{2}\big)\big\Vert_{L_q^{2}[0,1]}^2
 = \int_{0}^{1}xJ_{\nu}^2\big(qj_{m\nu}x;q^{2}\big){\rm d}_{q}x \\
\hphantom{\eta_{m,\nu}}{} = \frac{q-1}{2j_{m\nu}}q^{\nu-2}J_{\nu}\big(qj_{m\nu};q^{2}\big)J_{\nu}^{\prime}\big(j_{m\nu};q^{2}\big).
\end{gather*}
The sequence $\{u_m\}_{m}$ defines a system of functions which is orthonormal with respect to the inner product defined in the $L_q^{2}[0,1]$ space by~(\ref{inner-product}) and with the norm $\Vert \cdot\Vert_{L_q^{2}[0,1]}$ induced by it.

In this context, we are able to state the following analogue of the Riemann--Lebesgue theorem, based on a indirect proof within the scope of the inner product spaces: \emph{if $f\in L_q^2[0,1]$ then}
\begin{gather*}\lim_{m\to\infty}\int_{0}^{1}tf(t)J_{\nu}\big(qj_{m\nu}t;q^{2}\big){\rm d}_{q}t=0 .\end{gather*}
This is true since the sequence $\{u_m\}_{m}$ is orthonormal with respect to the inner product space $L_q^{2}[0,1]$, thus the corresponding proof can be carried out like in the classical case \cite[Corollary~36.4, p.~118]{WM}), being a consequence of the Bessel's inequality.

Alternatively, with a direct approach, we can extend the set of functions which satisfy the above property and state the following $q$-analogue of the Riemann--Lebesgue theorem.
\begin{Theorem}\label{Riemann--Lebesgue-2}
If $t^{\frac 1 2}f(t)\in L_q^2[0,1]$ then
\begin{gather*}\lim_{m\to\infty}\int_{0}^{1}tf(t)J_{\nu}\big(qj_{m\nu}t;q^{2}\big){\rm d}_{q}t=0 .\end{gather*}
\end{Theorem}
\begin{proof} Starting from the inner product (\ref{inner-product}) and then using the $q$-type H\"older inequality of \cite[Theorem~3.4, p.~346]{CP} with $p=2$, i.e., a $q$-type Cauchy--Schwartz inequality, we may write
\begin{gather}
\left|\int_{0}^{1}tf(t)J_{\nu}\big(qj_{m\nu}t;q^{2}\big){\rm d}_{q}t\right|
 \leq \left(\int_{0}^{1}t|f(t)|^2{\rm d}_{q}t\right)^{\frac 1 2} \left(\int_{0}^{1}tJ_{\nu}^2\big(qj_{m\nu}t;q^{2}\big){\rm d}_{q}t\right)^{\frac
1 2}\nonumber \\
 \hphantom{\left|\int_{0}^{1}tf(t)J_{\nu}\big(qj_{m\nu}t;q^{2}\big){\rm d}_{q}t\right|}{}
 = \left(\int_{0}^{1}t|f(t)|^2{\rm d}_{q}t\right)^{\frac 1 2}\eta_{m\nu}^{\frac 1 2},\label{C-S}
\end{gather}
where, by (\ref{eta}), $\eta_{m\nu}=\frac{q-1}{2j_{m\nu}}q^{\nu-2} J_{\nu}\big(qj_{m\nu};q^{2}\big)J_{\nu}^{\prime}\big(j_{m\nu};q^{2}\big)$.

In the expression for $\eta_{m\nu}$, we already control the asymptotic behavior, as $m\to\infty$, of all its factors.

Thus, joining Theorem~\ref{TheoremA} and~(\ref{asymptoticbehavior}), together with Theorems~\ref{AsymptoticJ3-at-shifted-zeros} and~\ref{asymptotic-derivativeJ3-at-zeros}, we obtain
\begin{gather}\label{asymptotic-eta}
\eta_{m\nu}=O\big(q^{2m}\big) ,\qquad \text{as} \quad m\to\infty .
\end{gather}
Finally, using in (\ref{C-S}) the hypothesis $t^{\frac 1 2}f(t)\in L_q^2[0,1]$ and the asymptotic relation~(\ref{asymptotic-eta}), it follows
\begin{gather*}
\int_{0}^{1}tf(t)J_{\nu}\big(qj_{m\nu}t;q^{2}\big){\rm d}_{q}t=O\big(q^{m}\big) ,\qquad \text{as}\quad m\to\infty ,
\end{gather*}
which proves this version of the Riemann--Lebesgue theorem.
\end{proof}

\begin{Remark} We emphasize that (\ref{asymptotic-eta}) implies that
\begin{gather*}\big\Vert x^{\frac{1}{2}}J_{\nu }\big(j_{m\nu }qx;q^{2}\big)\big\Vert_{L_q^{2}[0,1]}
=\sqrt{\eta_{m\nu}}=O\big(q^{m}\big) ,\qquad \text{as}\quad m\to\infty ,\end{gather*} hence
\begin{gather*}\lim_{m\to\infty} \big\Vert x^{\frac{1}{2}}J_{\nu }\big(j_{m\nu }qx;q^{2}\big)\big\Vert_{L_q^{2}[0,1]}=0 .\end{gather*}
\end{Remark}

However, for the following version of the classical Riemann--Lebesgue theorem: \emph{if $f$ is Riemann or Lebesgue integrable in $[a,b]$ then}
\begin{gather*}
\lim_{\mu\to\infty}\int_{a}^{b}f(t)\sin{(\mu t)}{\rm d}t=0 ,\qquad
\lim_{\mu\to\infty}\int_{a}^{b}f(t)\cos{(\mu t)}{\rm d}t=0 ,
\end{gather*}
we do not expect to prove a similar version for the case of the basic Fourier--Bessel expansions since we do not have, in this context, the nice properties and the formulary that the classical trigonometric functions satisfy, i.e., we do not expect to prove that
\begin{gather*}\lim_{\mu\to\infty}\int_{0}^{1}tf(t)J_{\nu}\big(\mu t;q^{2}\big){\rm d}_{q}t=0, \end{gather*}
when $f\in L_q[0,1]$ or $t^{\frac 1 2}f(t)\in L_q[0,1]$.

Some of the main reasons for that possible failure rely on the fact that, in the proof of the classical Riemann--Lebesgue theorem, involving the classical trigonometric functions, it is used the fact that these functions are bounded as well as some other known properties of the sine and cosine functions.

In this last direction, for the classical Bessel function $J_{\nu}(x)$, since it fails to satisfy the properties of the trigonometric function, an analogue of the Riemann--Lebesgue theorem \cite[p.~589]{Watson} was proved, not for the Bessel function $J_{\nu}(x)$ itself but for the function
\begin{gather*} T_n(t,x)=\sum_{m=1}^{n}\frac{2J_{\nu}\left(j_m x\right)J_{\nu}(j_m t)}{J_{\nu+1}^2(j_m)} ,\end{gather*}
where $j_m$, $m=1,2,3,\dots$, denote the positive zeros of the Bessel function $J_{\nu}(x)$ arranged in ascendent order of magnitude and $0<x\leq 1$, $0\leq t\leq 1$, $\nu\geq -\frac 1 2$:
\textit{if $\int_a^bt^{\frac 1 2}f(t){\rm d}t$ exists and is absolutely convergent then}
\begin{gather*}\lim_{n\to\infty}\int_a^btf(t)T_n(t,x){\rm d}t=0 ,\qquad 0<x\leq 1\quad \text{with} \quad a<b \quad \text{and} \quad a,b\in (0,1) .
\end{gather*} With this regard see also \cite[Remark~4, p.~13]{JLC3}.

\subsection*{Acknowledgements}
The author wants to thank the unknown referees for the valuable comments and remarks that helped to improve the paper. The author is also grateful to Professor Jos\'e Carlos Petronilho from CMUC (University of Coimbra) and Professor Renato \'Alvarez-Nodarse (University of Sevilla) for the valuables discussions. This research was partially supported  by   FCT - Funda\c{c}\~{a}o para a~Ci\^{e}ncia e a Tecnologia, within the project UID-MAT-00013/2013.

\pdfbookmark[1]{References}{ref}
\LastPageEnding


\begin{thebibliography}{99}
\footnotesize\itemsep=0pt

\bibitem{A}
Abreu L.D., Completeness, special functions and uncertainty principles over
 {$q$}-linear grids, \href{https://doi.org/10.1088/0305-4470/39/47/004}{\textit{J.~Phys.~A: Math. Gen.}} \textbf{39} (2006),
 14567--14580, \href{https://arxiv.org/abs/math.CA/0602440}{math.CA/0602440}.

\bibitem{A2}
Abreu L.D., Functions {$q$}-orthogonal with respect to their own zeros,
 \href{https://doi.org/10.1090/S0002-9939-06-08285-2}{\textit{Proc. Amer. Math. Soc.}} \textbf{134} (2006), 2695--2701.

\bibitem{AB}
Abreu L.D., Bustoz J., On the completeness of sets of {$q$}-{B}essel functions
 {$J_\nu^{(3)}(x;q)$}, in Theory and Applications of Special Functions,
 \href{https://doi.org/10.1007/0-387-24233-3_2}{\textit{Dev. Math.}}, Vol.~13, Springer, New York, 2005, 29--38.

\bibitem{ABC}
Abreu L.D., Bustoz J., Cardoso J.L., The roots of the third {J}ackson
 {$q$}-{B}essel function, \href{https://doi.org/10.1155/S016117120320613X}{\textit{Int.~J. Math. Math. Sci.}} (2003),
 4241--4248.

\bibitem{AAR}
Andrews G.E., Askey R., Roy R., Special functions, \href{https://doi.org/10.1017/CBO9781107325937}{\textit{Encyclopedia of
 Mathematics and its Applications}}, Vol.~71, Cambridge University Press, Cambridge, 1999.

\bibitem{Ann}
Annaby M.H., {$q$}-type sampling theorems, \href{https://doi.org/10.1007/BF03322983}{\textit{Results Math.}} \textbf{44}
 (2003), 214--225.

\bibitem{AM}
Annaby M.H., Mansour Z.S., On the zeros of the second and third {J}ackson
 {$q$}-{B}essel functions and their associated {$q$}-{H}ankel transforms,
 \href{https://doi.org/10.1017/S0305004109002357}{\textit{Math. Proc. Cambridge Philos. Soc.}} \textbf{147} (2009), 47--67.

\bibitem{BBEB}
Bettaibi N., Bouzeffour F., Ben~Elmonser H., Binous W., Elements of harmonic
 analysis related to the third basic zero order {B}essel function,
 \href{https://doi.org/10.1016/j.jmaa.2008.01.006}{\textit{J.~Math. Anal. Appl.}} \textbf{342} (2008), 1203--1219.

\bibitem{BC}
Bustoz J., Cardoso J.L., Basic analog of {F}ourier series on a {$q$}-linear
 grid, \href{https://doi.org/10.1006/jath.2001.3599}{\textit{J.~Approx. Theory}} \textbf{112} (2001), 134--157.

\bibitem{BS}
Bustoz J., Suslov S.K., Basic analog of {F}ourier series on a {$q$}-quadratic
 grid, \href{https://doi.org/10.4310/MAA.1998.v5.n1.a1}{\textit{Methods Appl. Anal.}} \textbf{5} (1998), 1--38,
 \href{https://arxiv.org/abs/math.CA/9706216}{math.CA/9706216}.

\bibitem{JLC}
Cardoso J.L., Basic {F}ourier series in a {$q$}-linear grid: convergence
 theorems, \href{https://doi.org/10.1016/j.jmaa.2005.10.043}{\textit{J.~Math. Anal. Appl.}} \textbf{323} (2006), 313--330.

\bibitem{JLC2}
Cardoso J.L., Basic {F}ourier series: convergence on and outside the
 {$q$}-linear grid, \href{https://doi.org/10.1007/s00041-010-9161-2}{\textit{J.~Fourier Anal. Appl.}} \textbf{17} (2011),
 96--114, \href{https://arxiv.org/abs/math.CA/0605764}{math.CA/0605764}.

\bibitem{JLC3}
Cardoso J.L., A few properties of the third {J}ackson {$q$}-{B}essel function,
 \href{https://doi.org/10.1007/s10476-016-0402-8}{\textit{Anal. Math.}} \textbf{42} (2016), 323--337.

\bibitem{CP}
Cardoso J.L., Petronilho J., Variations around {J}ackson's quantum operator,
 \href{https://doi.org/10.4310/MAA.2015.v22.n4.a1}{\textit{Methods Appl. Anal.}} \textbf{22} (2015), 343--358.

\bibitem{C}
Cie\'sli\'nski J.L., Improved {$q$}-exponential and {$q$}-trigonometric
 functions, \href{https://doi.org/10.1016/j.aml.2011.06.009}{\textit{Appl. Math. Lett.}} \textbf{24} (2011), 2110--2114,
 \href{https://arxiv.org/abs/1006.5652}{arXiv:1006.5652}.

\bibitem{ESF}
Elmonser H., Sellami M., Fitouhi A., Inequalities related to the third
 {J}ackson {$q$}-{B}essel function of order zero, \href{https://doi.org/10.1186/1029-242X-2013-289}{\textit{J.~Inequal. Appl.}}
 \textbf{2013} (2013), 2013:289, 22~pages.

\bibitem{E1978}
Exton H., A basic analogue of the {B}essel--{C}lifford equation,
 \textit{J\~n\=an\=abha} \textbf{8} (1978), 49--56.

\bibitem{E}
Exton H., {$q$}-hypergeometric functions and applications, \textit{Ellis Horwood
 Series: Mathematics and its Applications}, Ellis Horwood Ltd., Chichester,
 Halsted Press, New York, 1983.

\bibitem{GR}
Gasper G., Rahman M., Basic hypergeometric series, \href{https://doi.org/10.1017/CBO9780511526251}{\textit{Encyclopedia of
 Mathematics and its Applications}}, Vol.~35, Cambridge University Press,
 Cambridge, 1990.

\bibitem{Ha}
Hardy G.H., Notes on special systems of orthogonal functions~({II}): on
 functions orthogonal with respect to their own zeros, \href{https://doi.org/10.1112/jlms/s1-14.1.37}{\textit{J.~London Math.
 Soc.}} \textbf{14} (1939), 37--44.

\bibitem{H}
Hayman W.K., On the zeros of a {$q$}-{B}essel function, in Complex analysis and
 dynamical systems~{II}, \href{https://doi.org/10.1090/conm/382/07060}{\textit{Contemp. Math.}}, Vol.~382, Amer. Math. Soc.,
 Providence, RI, 2005, 205--216.

\bibitem{KS}
Koelink H.T., Swarttouw R.F., On the zeros of the {H}ahn--{E}xton
 {$q$}-{B}essel function and associated {$q$}-{L}ommel polynomials,
 \href{https://doi.org/10.1006/jmaa.1994.1327}{\textit{J.~Math. Anal. Appl.}} \textbf{186} (1994), 690--710.

\bibitem{KooS}
Koornwinder T.H., Swarttouw R.F., On {$q$}-analogues of the {F}ourier and
 {H}ankel transforms, \href{https://doi.org/10.2307/2154118}{\textit{Trans. Amer. Math. Soc.}} \textbf{333} (1992),
 445--461.

\bibitem{D}
Olde~Daalhuis A.B., Asymptotic expansions for {$q$}-gamma, {$q$}-exponential,
 and {$q$}-{B}essel functions, \href{https://doi.org/10.1006/jmaa.1994.1339}{\textit{J.~Math. Anal. Appl.}} \textbf{186}
 (1994), 896--913.

\bibitem{S}
Suslov S.K., ``{A}ddition'' theorems for some {$q$}-exponential and
 {$q$}-trigonometric functions, \href{https://doi.org/10.4310/MAA.1997.v4.n1.a2}{\textit{Methods Appl. Anal.}} \textbf{4}
 (1997), 11--32.

\bibitem{S2003}
Suslov S.K., An introduction to basic {F}ourier series, \href{https://doi.org/10.1007/978-1-4757-3731-8}{\textit{Developments in
 Mathematics}}, Vol.~9, Kluwer Academic Publishers, Dordrecht, 2003.

\bibitem{S-PhD}
Swarttouw R.F., The {H}ahn--{E}xton $q$-{B}essel function, Ph.D.~Thesis,
 Technische Universiteit Delft, 1992.

\bibitem{St}
\v{S}tampach F., Nevanlinna extremal measures for polynomials related to
 {$q^{-1}$}-{F}ibonacci polynomials, \href{https://doi.org/10.1016/j.aam.2016.02.005}{\textit{Adv. in Appl. Math.}} \textbf{78}
 (2016), 56--75.

\bibitem{SS2016}
\v{S}tampach F., \v{S}\v tov\'{\i}\v{c}ek P., The {N}evanlinna parametrization
 for {$q$}-{L}ommel polynomials in the indeterminate case, \href{https://doi.org/10.1016/j.jat.2015.09.002}{\textit{J.~Approx.
 Theory}} \textbf{201} (2016), 48--72, \href{https://arxiv.org/abs/1407.0217}{arXiv:1407.0217}.

\bibitem{Watson}
Watson G.N., A treatise on the theory of {B}essel functions, Cambridge
 University Press, Cambridge, The Macmillan Company, New York, 1944.

\bibitem{WM}
Wilcox H.J., Myers D.L., An introduction to {L}ebesgue integration and
 {F}ourier series, Dover Publications, Inc., New York, 1994.

\end{thebibliography}
\end{document}